\documentclass{amsart}  

\usepackage{hyperref}

\newcommand*{\mailto}[1]{\href{mailto:#1}{\nolinkurl{#1}}}

\newtheorem{theorem}{Theorem}[section]
\newtheorem{definition}[theorem]{Definition}
\newtheorem{lemma}[theorem]{Lemma}
\newtheorem{proposition}[theorem]{Proposition}
\newtheorem{corollary}[theorem]{Corollary}

\newtheorem{remark}[theorem]{Remark}

\newcommand{\R}{{\mathbb R}}
\newcommand{\N}{{\mathbb N}}

\newcommand{\C}{{\mathbb C}}

\newcommand{\Wr}{\mathsf{w}}
\newcommand{\E}{\mathrm{e}}
\newcommand{\I}{\mathrm{i}}
\newcommand{\supp}{\mathrm{supp}}

\newcommand{\loc}{\mathrm{loc}}
\newcommand{\cc}{\mathrm{c}}

\newcommand{\redot}{\cdot\,}

\newcommand{\OO}{\mathcal{O}}
\newcommand{\oo}{o}

\newcommand{\dip}{\upsilon}

\numberwithin{equation}{section}

\begin{document}

\title[Generalized indefinite strings of Stieltjes type]{Continued fraction expansions of Herglotz--Nevanlinna functions and generalized indefinite strings of Stieltjes type}

\author[J.\ Eckhardt]{Jonathan Eckhardt}
\address{Department of Mathematical Sciences\\ Loughborough University\\ Loughborough LE11 3TU \\ UK}
\email{\mailto{J.Eckhardt@lboro.ac.uk}}


\thanks{{\it Research supported by the Austrian Science Fund (FWF) under Grant No.\ P29299}}

\keywords{Continued fraction expansions, generalized indefinite strings, Stieltjes type}
\subjclass[2020]{Primary 30B70, 34A55; Secondary 34L05, 34B20}  

\begin{abstract}
 We employ some results about continued fraction expansions of Herglotz--Nevanlinna functions to characterize the spectral data of generalized indefinite strings of Stieltjes type.   
 In particular, this solves the corresponding inverse spectral problem through explicit formulas. 
\end{abstract}

\maketitle

\section*{Introduction}

Stieltjes continued fractions played a decisive role in the solution of the inverse spectral problem for Krein strings  \cite{dymc76, kakr74, ka94, kowa82, kr52}. 
 A certain modification of these continued fractions is of the same relevance for generalized indefinite strings, a class of spectral problems introduced in \cite{IndefiniteString}, based on previous work in \cite{krla79, krla80, la76, lawi98}.
 This kind of continued fractions arose in \cite{krla79, krla80} in connection with indefinite analogues of moment problems, was further studied in \cite{deko15, deko17} and applied to conservative multi-peakon solutions of the Camassa--Holm equation in \cite{ConservMP}. 
 It is the purpose of this article to discuss under which conditions a general Herglotz--Nevanlinna function admits a continued fraction expansion of this form.
 This will be done in the first section, which is close to classical material in  \cite{ak65} and \cite{krla79, krla80}, but does not seem to be available in the desired form.
 Subsequently, we will then use these findings in the second section to characterize the spectral data of generalized indefinite strings of Stieltjes type, whose coefficients are supported on discrete sets (near the left endpoint).  
This kind of generalized indefinite strings is closely related to Hamburger Hamiltonians \cite{ka02} for canonical systems and hence also connected to the classical moment problem \cite{IndMoment}.
 In particular, the results provide a solution of the inverse spectral problem for generalized indefinite strings of Stieltjes type with explicit formulas for the solution in terms of the moments of the spectral measure.
 A special case of such an inverse problem for indefinite strings was recently solved in \cite{flwi17} by means of a somewhat different approach.

\section{Continued fraction expansions of Herglotz--Nevanlinna functions}
 
 Let $m$ be a Herglotz--Nevanlinna function,  that is, the function $m$ is defined and analytic on $\C\backslash\R$, maps the upper complex half-plane into the closure of the upper complex half-plane and satisfies the symmetry relation
 \begin{align}
   m(z)^\ast = m(z^\ast), \quad z\in\C\backslash\R.
 \end{align}
 It is well known that such a function admits an integral representation  
 of the form
 \begin{align}\label{eqnmInt}
   m(z)= a + bz + \int_\R \frac{1}{\lambda-z} - \frac{\lambda}{1+\lambda^2}d\rho(\lambda),\quad z\in\C\backslash\R,
 \end{align}
 where $a$ is a real constant, $b$ is a non-negative constant and $\rho$ is a non-negative Borel measure on $\R$ which is subject to the growth restriction 
 \begin{align}
   \int_\R \frac{d\rho(\lambda)}{1+\lambda^2} < \infty. 
 \end{align}
 The constants $a$ and $b$, as well as the measure $\rho$ in this integral representation are uniquely determined by the function $m$ and may be recovered explicitly. 
 
 As long as they exist, we will denote with $s_0,s_1,\ldots$ the moments of the measure $\rho$, that is, we set 
 \begin{align}
   s_k = \int_\R \lambda^k d\rho(\lambda). 
 \end{align}
 It follows readily from expanding the integrand in the representation~\eqref{eqnmInt} that for each $K\in\N_0$, the function $m$ allows the asymptotic expansion 
 \begin{align}
    - m(z) = s_{-2}z + s_{-1} + \sum_{k=0}^{2K} \frac{s_k}{z^{k+1}} + \oo\biggl(\frac{1}{|z|^{2K+1}}\biggr)
 \end{align}
 as $|z|\rightarrow\infty$ along the imaginary axis, provided the moments of the measure $\rho$ exist up to order $2K$.
 Here, the real constants $s_{-2}$ and $s_{-1}$ are defined by 
 \begin{align}
   s_{-2} & = - b, & s_{-1} & =  \int_\R \frac{\lambda}{1+\lambda^2} d\rho(\lambda) - a. 
 \end{align}
 In fact, the converse of this statement holds true as well; see \cite[Theorem~3.2.1]{ak65}. 
 
 \begin{theorem}\label{thmmmomexp}
   For any fixed $K\in\N_0$, the moments of the measure $\rho$ exist up to order $2K$ if and only if the function $m$ allows the asymptotic expansion
   \begin{align}\label{eqnmexpinf}
       -m(z) = s_{-2}z  + s_{-1} + \sum_{k=0}^{2K} \frac{s_k}{z^{k+1}} + \oo\biggl(\frac{1}{|z|^{2K+1}}\biggr)
   \end{align}
   as $|z|\rightarrow\infty$ along the imaginary axis for some real constants $s_{-2},\ldots,s_{2K}$.
 \end{theorem}
 
 We are going to demonstrate below that these two conditions are further equivalent to a particular continued fraction expansion of the function $m$. 
 To this end, we first introduce the Hankel determinants  
 \begin{align} 
  \Delta_{0,k} & = \begin{vmatrix} s_0 & s_1 & \cdots & s_{k-1} \\ s_1 & s_2 & \cdots & s_{k} \\ \vdots & \vdots & \ddots & \vdots \\ s_{k-1} & s_{k} & \cdots & s_{2k-2} \end{vmatrix}, 
  \end{align}
 as far as they are well-defined. 
  Since the entries are the moments of a non-negative measure, it is readily verified that these determinants are certainly non-negative.   
  Apart from this, we also introduce the Hankel determinants  
   \begin{align}  
  \Delta_{1,k} & = \begin{vmatrix} s_1 & s_2 & \cdots & s_{k} \\ s_2 & s_3 & \cdots & s_{k+1} \\ \vdots & \vdots & \ddots & \vdots \\ s_{k} & s_{k+1} & \cdots & s_{2k-1} \end{vmatrix}, & 
  \Delta_{2,k} & = \begin{vmatrix} s_2 & s_3 & \cdots & s_{k+1} \\ s_3 & s_4 & \cdots & s_{k+2} \\ \vdots & \vdots & \ddots & \vdots \\ s_{k+1} & s_{k+2} & \cdots & s_{2k} \end{vmatrix},
  \end{align}
   as well as 
      \begin{align}
  \Delta_{-1,k} & = \begin{vmatrix} s_{-1} & s_0 & \cdots & s_{k-2} \\ s_0 & s_1 & \cdots & s_{k-1} \\ \vdots & \vdots & \ddots & \vdots \\ s_{k-2} & s_{k-1} & \cdots & s_{2k-3} \end{vmatrix}, & 
    \Delta_{-2,k} & = \begin{vmatrix} s_{-2} & s_{-1} & \cdots & s_{k-3} \\ s_{-1} & s_0 & \cdots & s_{k-2} \\ \vdots & \vdots & \ddots & \vdots \\ s_{k-3} & s_{k-2} & \cdots & s_{2k-4} \end{vmatrix},
 \end{align}
 again, provided they exist. 
 In order to avoid ambiguity, it should be pointed out that all these determinants have to be interpreted as equal to one when $k$ is zero.  
 For future reference, we also state the useful relations 
  \begin{align} 
   \label{eqnDeltaRel}   \Delta_{1,k}  \Delta_{-1,k}  - \Delta_{1,k-1} \Delta_{-1,k+1} & = \Delta_{0,k}^2,  \\       
   \label{eqnDeltaRelp}  \Delta_{2,k}  \Delta_{0,k}  - \Delta_{2,k-1} \Delta_{0,k+1} & =  \Delta_{1,k}^2,   \\         
      \label{eqnDeltaRelm}         \Delta_{0,k}  \Delta_{-2,k}  - \Delta_{0,k-1} \Delta_{-2,k+1} & =  \Delta_{-1,k}^2,   
       \end{align} 
 which follow from Sylvester's determinant identity \cite{ba68, ga59} and hold as long as the respective determinants are well-defined.

Let us suppose for now that $m$ is a rational Herglotz--Nevanlinna function and denote with $D$ the number of poles of $m$. 
 Because the support of the measure $\rho$ coincides with the poles of $m$, it is evident that all moments of the measure $\rho$ exist in this case. 
 Moreover, it follows that the determinants $\Delta_{0,0},\ldots,\Delta_{0,D}$ are positive but $\Delta_{0,k}$ is zero when $k>D$.  
 Similarly, one sees that $\Delta_{1,k}$ is zero when $k>D$ and that $\Delta_{1,D}$ is zero if and only if zero is a pole of $m$. 
 We then define $N\in\N_0$ such that the number of non-zero elements of the sequence $\Delta_{1,0},\ldots,\Delta_{1,D}$ is exactly $N+1$ and introduce the function 
 \begin{align}\label{eqnkapparat}
   \kappa:\{1,\dots,N+1\}\rightarrow\{0,\ldots,D\}
 \end{align}
  such that $\kappa(n)$ is the smallest integer $k\in\N_0$ for which the sequence $\Delta_{1,0},\ldots,\Delta_{1,k}$ has precisely $n$ non-zero elements. 
  One observes that the increasing function $\kappa$ 
  is defined in such a way that $\Delta_{1,\kappa(1)},\ldots,\Delta_{1,\kappa(N+1)}$ enumerates all non-zero members of the sequence $\Delta_{1,0},\ldots,\Delta_{1,D}$.  
 As it follows from relation~\eqref{eqnDeltaRel} that there are no consecutive zeros in the sequence $\Delta_{1,0},\ldots,\Delta_{1,D}$, we may conclude that 
  \begin{align}
       \kappa(n+1) = \begin{cases} \kappa(n)+1, & \Delta_{1,\kappa(n)+1}\not=0, \\ \kappa(n)+2, & \Delta_{1,\kappa(n)+1}=0, \end{cases}
  \end{align} 
 for all $n\in\{1,\ldots,N\}$. 
 Since we have $\kappa(1)=0$, this determines $\kappa$ recursively. 
   
  \begin{proposition}\label{propmrat}
   If $m$ is a rational Herglotz--Nevanlinna function, then it admits the continued fraction expansion
       \begin{align}\label{eqnmratCF}
   m(z) = \dip_0 z + \omega_0 + \cfrac{1}{-l_{1}z + \cfrac{1}{\dip_1 z + \omega_1 + \cfrac{1}{ \;\ddots\;  + \cfrac{1}{-l_{N}z + \cfrac{1}{\dip_{N} z + \omega_{N} - \cfrac{r}{z}}}}}}\,, \quad z\in\C\backslash\R,
 \end{align}
  where the non-negative constants  $\dip_0,\ldots,\dip_N$ and the real constants $\omega_0,\ldots,\omega_{N}$ are given by    
  \begin{subequations}\label{eqnomegadipnDel}
  \begin{align}
     \dip_0 & = -\frac{\Delta_{-2,1}}{\Delta_{0,0}}, & \dip_n & =   \frac{\Delta_{-2,\kappa(n)+2}}{\Delta_{0,\kappa(n)+1}} -  \frac{\Delta_{-2,\kappa(n+1)+1}}{\Delta_{0,\kappa(n+1)}}, \label{eqndipnDel} \\
      \omega_0 & = -\frac{\Delta_{-1,1}}{\Delta_{1,0}}, & \omega_n & = \frac{\Delta_{-1,\kappa(n)+1}}{\Delta_{1,\kappa(n)}} - \frac{\Delta_{-1,\kappa(n+1)+1}}{\Delta_{1,\kappa(n+1)}},  \label{eqnomeganDel} 
     \end{align}
     \end{subequations}
    and the positive constants $l_1,\ldots,l_N$ and the non-negative constant $r$ are given by
     \begin{align}
      l_n &  =  \frac{\Delta_{1,\kappa(n)}^2}{\Delta_{0,\kappa(n)}\Delta_{0,\kappa(n)+1}},  & r & = \frac{\Delta_{0,\kappa(N+1)}\Delta_{0,\kappa(N+1)+1}}{\Delta_{1,\kappa(N+1)}^2}.  \label{eqnlnDel}
  \end{align}
\end{proposition}

\begin{proof}
  Suppose that $m$ is a rational Herglotz--Nevanlinna function. 
  For each $t\in\R$, we define the rational Herglotz--Nevanlinna function $m_t$ by
  \begin{align*}
    m_t(z) = -s_{-2} z - s_{-1} + \int_{\R} \frac{1}{\lambda-z} \E^{t\lambda} d\rho(\lambda), \quad z\in\C\backslash\R.
  \end{align*}
 The moments and Hankel determinants corresponding to the function $m_t$ will be denoted in a natural way with an additional subscript $t$. 
 In particular, notice that $m_0$ coincides with our initial function $m$ and hence so do the associated quantities (which is why we will omit the additional subscripts in this case). 
 As each of the Hankel determinants depends analytically on $t$, we may conclude that the determinants $\Delta_{t,1,0},\ldots,\Delta_{t,1,\kappa(N+1)}$ are all non-zero as long as $t\not=0$ is small enough. 
 Indeed, even if $\Delta_{1,k}$ does vanish for some $k\in\lbrace1,\ldots,\kappa(N+1)-1\rbrace$ (notice here that $\Delta_{1,\kappa(N+1)}$ is certainly non-zero), then this holds because the derivative 
 \begin{align*}
   \left[\frac{d}{dt} \Delta_{t,1,k}\right]_{t=0} =  \begin{vmatrix} {s}_1 & s_2 & \cdots & s_{k-1} & s_{k+1} \\ s_2 & s_3  & \cdots & s_k & s_{k+2} \\ \vdots & \vdots & \ddots & \vdots & \vdots \\ s_{k} & s_{k+1}  & \cdots & s_{2k-2} & s_{2k}  \end{vmatrix}  =: \Delta_{1,k}'
 \end{align*}
 is non-zero (here it suffices to observe that we have 
 \begin{align*}
  \frac{d}{dt} s_{t,k} & = s_{t,k+1}, \quad  k\in\N_0,
 \end{align*}
 in order to verify the expression for the  derivative). 
 More precisely, non-vanishing of the above determinant follows from the relation  
 \begin{align*}
  \Delta_{1,k}' \Delta_{0,k} = \Delta_{0,k+1}\Delta_{1,k-1},
 \end{align*}
 which is a consequence of Sylvester's determinant identity \cite{ba68, ga59} (when $\Delta_{1,k}=0$). 
 
 It follows that, as long as $t\not=0$ is small enough, the function $m_t$ admits a Stieltjes continued fraction expansion \cite{st94, st95} (see \cite[Theorem~1.39]{hoty12}) of the form 
        \begin{align*}
   m_t(z) =-s_{-2} z - s_{-1} + \cfrac{1}{-l_{t,1}z + \cfrac{1}{\omega_{t,1} + \cfrac{1}{ \;\ddots\;  + \cfrac{1}{-l_{t,K}z + \cfrac{1}{\omega_{t,K} - \cfrac{r_t}{z}}}}}}\,, \quad z\in\C\backslash\R,
 \end{align*}
  where $K = \kappa(N+1)$, the non-zero real constants $\omega_{t,1},\ldots,\omega_{t,K}$  are given by
      \begin{align}
\label{eqnCobnInv}  \omega_{t,k} & = \frac{\Delta_{t,0,k}^2}{\Delta_{t,1,k-1} \Delta_{t,1,k}} = \frac{\Delta_{t,-1,k}}{\Delta_{t,1,k-1}} - \frac{\Delta_{t,-1,k+1}}{\Delta_{t,1,k}}
 \end{align}
and the positive constants $l_{t,1},\ldots,l_{t,K}$  as well as the non-negative constant $r_t$ are given by
    \begin{align}\label{eqnCoanInv}
  l_{t,k} & = \frac{\Delta_{t,1,k-1}^2}{\Delta_{t,0,k-1} \Delta_{t,0,k}},  & r_t & = \frac{\Delta_{t,0,K} \Delta_{t,0,K+1}}{\Delta_{t,1,K}^2}.
 \end{align}
 We note that the continued fraction has to be interpreted as  
 \begin{align*}
   m_t(z) = -s_{-2}z - s_{-1} - \frac{r_t}{z}, \quad z\in\C\backslash\R, 
 \end{align*}
 when $K$ is zero and that the alternative expression in~\eqref{eqnCobnInv} is obtained by using relation~\eqref{eqnDeltaRel}. 
 Put differently, the continued fraction expansion means that, upon defining the rational Herglotz--Nevanlinna functions $q_{t,0},\ldots,q_{t,K}$ recursively via  
 \begin{align}\label{eqnMrec}
  \frac{1}{q_{t,k}(z)} = -l_{t,k+1} z+\frac{1}{\omega_{t,k+1}+q_{t,k+1}(z)}, \quad z\in\C\backslash\R,
 \end{align}
 for every $k\in\{0,\ldots,K-1\}$, where $q_{t,K}$ is given by 
 \begin{align*}
  q_{t,K}(z) = -\frac{r_t}{z}, \quad z\in\C\backslash\R, 
 \end{align*}
 we eventually end up with  
 \begin{align}\label{eqnmtqf0}
   m_t(z) = -s_{-2} z -s_{-1} + q_{t,0}(z), \quad z\in\C\backslash\R. 
 \end{align}
 
 As $t\rightarrow0$, the functions $q_{t,\kappa(N+1)}=q_{t,K}$ clearly converge pointwise to the rational Herglotz--Nevanlinna function $q_{\kappa(N+1)}$ defined by  
 \begin{align}\label{eqnqkappaNp}
  q_{\kappa(N+1)}(z) = -\frac{r}{z}, \quad z\in\C\backslash\R,
 \end{align} 
 where the constant $r$ is given by~\eqref{eqnlnDel}. 
 Now let $n\in\lbrace 1,\ldots,N\rbrace$ and assume that the functions $q_{t,\kappa(n+1)}$ converge pointwise to some rational Herglotz--Nevanlinna function $q_{\kappa(n+1)}$. 
 If $\Delta_{1,\kappa(n)+1}$ does not vanish, then $\kappa(n+1)=\kappa(n)+1$ and we infer from~\eqref{eqnMrec} as well as the formulas~\eqref{eqnCobnInv} and~\eqref{eqnCoanInv} that the functions $q_{t,\kappa(n)}$ converge pointwise to the Herglotz--Nevanlinna function $q_{\kappa(n)}$ defined via 
 \begin{align}\label{eqnqkappan}
  \frac{1}{q_{\kappa(n)}(z)} & = -l_{n} z  + \cfrac{1}{\dip_{n} z + \omega_{n} + q_{\kappa(n+1)}(z)}\,, \quad z\in\C\backslash\R,
 \end{align}
 where  the constants $\dip_n$, $\omega_n$ and $l_n$ are given by~\eqref{eqnomegadipnDel} and~\eqref{eqnlnDel}; note that $\dip_n$ is zero here. 
 Otherwise, when $\Delta_{1,\kappa(n)+1}$ vanishes, one has $\kappa(n+1) = \kappa(n)+2$ and using the second expression in~\eqref{eqnCobnInv} we see that     
 \begin{align*}
  \omega_{t,\kappa(n)+1}+\omega_{t,\kappa(n)+2} & \rightarrow  \frac{\Delta_{-1,\kappa(n)+1}}{ \Delta_{1,\kappa(n)}} - \frac{\Delta_{-1,\kappa(n+1)+1}}{\Delta_{1,\kappa(n+1)}}  = \omega_{n}, 
  \end{align*}
  where the constant $\omega_n$ is given by~\eqref{eqnomeganDel}.
  Moreover, taking~\eqref{eqnCoanInv} into account, we also get 
  \begin{align*}
  l_{t,\kappa(n)+2}\omega_{t,\kappa(n)+1} & \rightarrow 0, \\
  l_{t,\kappa(n)+2}\omega_{t,\kappa(n)+2} & \rightarrow 0, \\
  l_{t,\kappa(n)+2}\omega_{t,\kappa(n)+1} \omega_{t,\kappa(n)+2} & \rightarrow \frac{\Delta_{0,\kappa(n)+1}\Delta_{0,\kappa(n)+2}}{\Delta_{1,\kappa(n)}\Delta_{1,\kappa(n)+2}}, 
 \end{align*}
 in this case. 
 Utilizing relation~\eqref{eqnDeltaRel} with $k=\kappa(n)+1$ as well as $k=\kappa(n)+2$ first and subsequently relation~\eqref{eqnDeltaRelm}, we infer that the last limit is actually equal to 
 \begin{align*}
   -\frac{\Delta_{-1,\kappa(n)+2}^2}{\Delta_{0,\kappa(n)+1}\Delta_{0,\kappa(n)+2}} =  \frac{\Delta_{-2,\kappa(n+1)+1}}{\Delta_{0,\kappa(n+1)}} - \frac{\Delta_{-2,\kappa(n)+2}}{\Delta_{0,\kappa(n)+1}} = - \dip_n, 
 \end{align*}
 where the constant $\dip_n$ is given by~\eqref{eqndipnDel}. 
 In particular, this shows that $\dip_n$ is positive as relation~\eqref{eqnDeltaRel} with $k=\kappa(n)+1$ shows that the numerator on the left-hand side is not zero. 
  Now observe that in the current case, the functions $q_{t,\kappa(n)}$ satisfy  
 \begin{align*}
   \frac{1}{q_{t,\kappa(n)}(z)}=-l_{t,\kappa(n)+1} z +\cfrac{1}{\omega_{t,\kappa(n)+1}+\cfrac{1}{-l_{t,\kappa(n)+2}z+\cfrac{1}{\omega_{t,\kappa(n)+2}+q_{t,\kappa(n+1)}(z)}}}
 \end{align*} 
 for every $z\in\C\backslash\R$. 
 With the help of the limits above, a computation then shows that the functions $q_{t,\kappa(n)}$ converge pointwise to the rational Herglotz--Nevanlinna function $q_{\kappa(n)}$ defined by~\eqref{eqnqkappan}, 
 where the constant $l_n$ is given by~\eqref{eqnlnDel}. 

 Concluding, we notice that the functions $m_t$ converge pointwise to our initial function $m$ by definition and thus we see from~\eqref{eqnmtqf0} that    
 \begin{align*}
   m(z) = \dip_0 z + \omega_0 + q_{\kappa(1)}(z), \quad z\in\C\backslash\R, 
 \end{align*}
  where the constants $\dip_0$ and $\omega_0$ are given by~\eqref{eqnomegadipnDel}. 
 In view of~\eqref{eqnqkappan} and~\eqref{eqnqkappaNp}, this shows that the function $m$ admits the claimed continued fraction expansion.
\end{proof}

It is not difficult to see that any continued fraction of the above form is a rational Herglotz--Nevanlinna function in turn. 
 We also note that the mere fact that every rational Herglotz--Nevanlinna function can be expanded in such a way is much simpler to prove (see \cite[Lemma~B]{UniSolCP}), whereas working out explicit formulas for the constants takes more effort. 

\begin{remark}\label{remCFcoefRexp}
The constants in the continued fraction in Proposition~\ref{propmrat} can also be expressed in different ways. 
 In view of relation~\eqref{eqnDeltaRelp}, the positive constants $l_1,\ldots,l_N$ may be written in the form
 \begin{align}\label{eqnlnDelalt}
  l_1 & = \frac{\Delta_{2,0}}{\Delta_{0,1}}, & l_n & =  \frac{\Delta_{2,\kappa(n+1)-1}}{\Delta_{0,\kappa(n+1)}} - \frac{\Delta_{2,\kappa(n)-1}}{\Delta_{0,\kappa(n)}}, \quad n>1.
 \end{align}
Apart from this, relation~\eqref{eqnDeltaRel} shows that we have 
\begin{align}
  \omega_n = & \frac{\Delta_{0,\kappa(n+1)}^2}{\Delta_{1,\kappa(n)}\Delta_{1,\kappa(n+1)}} \not= 0,  & \dip_n & = 0, 
\end{align}
 for $n\in\{1,\ldots,N\}$ as long as $\Delta_{1,\kappa(n)+1}$ is not zero. 
 On the other side, if $\Delta_{1,\kappa(n)+1}$ is zero, then relation~\eqref{eqnDeltaRelm} allows us to write 
 \begin{align}
   \dip_n = \frac{\Delta_{-1,\kappa(n)+2}^2}{\Delta_{0,\kappa(n)+1}\Delta_{0,\kappa(n)+2}}>0.
 \end{align}
 In particular, these expressions make it clear that the constant $\dip_n$ is not zero if and only if $\Delta_{1,\kappa(n)+1}$ vanishes and also that $\dip_n+|\omega_n|>0$ for all $n\in\{1,\ldots,N\}$.
\end{remark}

 Before we proceed to non-rational Herglotz--Nevanlinna functions, let us first provide two auxiliary results.
 In order to state them, let $m_1$ and $m_2$ be Herglotz--Nevanlinna functions and denote with $\rho_1$ and $\rho_2$ the corresponding measures in the respective integral representations.
 
   \begin{lemma}\label{lemMomRel}
   Let $K\in\N_0$ and suppose that  
   \begin{align}
     m_1(z) =  \frac{1}{- m_2(z)}, \quad z\in\C\backslash\R. 
   \end{align}

  1.\  Assume that $m_2(z)/z\rightarrow\beta$ for some positive $\beta$ as $|z|\rightarrow\infty$ along the imaginary axis. 
  Then the moments of the measure $\rho_2$ exist up to order $2K$ if and only if the moments of the measure $\rho_1$ exist up to order $2K+2$. 

 2.\ Assume that $m_1(z)/z\rightarrow 0$ and $m_2(z)/z\rightarrow0$ as $|z|\rightarrow\infty$ along the imaginary axis. 
  Then the moments of the measure $\rho_2$ exist up to order $2K$ if and only if the moments of the measure $\rho_1$ exist up to order $2K$.
 \end{lemma}
 
 \begin{proof}
   We begin with the case when $m_2(z)/z\rightarrow\beta$ for some positive $\beta$ as $|z|\rightarrow\infty$ along the imaginary axis.
   If the moments of the measure $\rho_2$ exist up to order $2K$, then from Theorem~\ref{thmmmomexp} we infer that
   \begin{align*}
    \frac{m_2(z)}{\beta z} =  1 - \sum_{k=-1}^{2K} \frac{s_{2,k}}{\beta z^{k+2}} + \oo\biggl(\frac{1}{|z|^{2K+2}}\biggr)
   \end{align*} 
   as $|z|\rightarrow\infty$ along the imaginary axis for some real constants $s_{2,-1},\ldots,s_{2,2K}$.
   This implies that 
   \begin{align*}
     - \beta z m_1(z) = \frac{\beta z}{m_2(z)} =   \biggl(1 - \sum_{k=-1}^{2K} \frac{s_{2,k}}{\beta z^{k+2}}\biggr)^{-1}  +\oo\biggl(\frac{1}{|z|^{2K+2}}\biggr) 
   \end{align*}
   as $|z|\rightarrow\infty$ along the imaginary axis. 
   Since the first term on the right-hand side is a rational function, we conclude from Theorem~\ref{thmmmomexp} that the moments of the measure $\rho_1$ exist up to order $2K+2$. 
   Conversely, if the moments of the measure $\rho_1$ exist up to order $2K+2$, then we have 
   \begin{align*}
     - \beta z m_1(z) = 1 + \sum_{k=1}^{2K+2} \frac{\beta s_{1,k}}{z^{k}} + \oo\biggl(\frac{1}{|z|^{2K+2}}\biggr)
   \end{align*}
   as $|z|\rightarrow\infty$ along the imaginary axis for some real constants $s_{1,1},\ldots,s_{1,2K+2}$ and thus 
   \begin{align*}
     \frac{m_2(z)}{\beta z} = \frac{1}{-\beta z m_1(z)} = \biggl(1+\sum_{k=1}^{2K+2} \frac{\beta s_{1,k}}{z^{k}}\biggr)^{-1} + \oo\biggl(\frac{1}{|z|^{2K+2}}\biggr)
   \end{align*}
   as $|z|\rightarrow\infty$ along the imaginary axis. 
   Like before, we may conclude again that the moments of the measure $\rho_2$ exist up to order $2K$ by invoking  Theorem~\ref{thmmmomexp}. 
   
   Now let us assume that $m_1(z)/z\rightarrow0$ and $m_2(z)/z\rightarrow0$ as $|z|\rightarrow\infty$ along the imaginary axis.
   If the moments of the measure $\rho_2$ exist up to order $2K$, then from Theorem~\ref{thmmmomexp} we infer that 
      \begin{align*}
      -m_2(z) = s_{2,-1} + \sum_{k=0}^{2K} \frac{s_{2,k}}{z^{k+1}} + \oo\biggl(\frac{1}{|z|^{2K+1}}\biggr)
   \end{align*} 
   as $|z|\rightarrow\infty$ along the imaginary axis for some real constants $s_{2,-1},\ldots,s_{2,2K}$.
   Since $m_1(z)/z\rightarrow0$  as $|z|\rightarrow\infty$ along the imaginary axis, we conclude that $s_{2,-1}$ is not zero 
   and thus we have  
   \begin{align*}
     m_1(z) = \frac{1}{-m_2(z)} =  \biggl(s_{2,-1} + \sum_{k=0}^{2K} \frac{s_{2,k}}{z^{k+1}}\biggr)^{-1}  +\oo\biggl(\frac{1}{|z|^{2K+1}}\biggr) 
   \end{align*}
   as $|z|\rightarrow\infty$ along the imaginary axis. 
   As the first term on the right-hand side is a rational function, we infer from Theorem~\ref{thmmmomexp} that the moments of the measure $\rho_1$ exist up to order $2K$. 
  The converse direction follows by symmetry. 
 \end{proof}
 
 \begin{lemma}\label{lemHNclose}
   If there are Herglotz--Nevanlinna functions $\tilde{m}_1$ and $\tilde{m}_2$ with  
       \begin{align}
      \lim_{\eta\rightarrow\infty} \frac{1}{\I\eta} \biggl(\frac{1}{\tilde{m}_1(\I\eta)} - \frac{1}{\tilde{m}_2(\I\eta)} \biggr) = 0
    \end{align}
   such that the functions $m_i$ for $i\in\{1,2\}$ admit the continued fraction expansion 
    \begin{align}
   m_i(z) =  \dip_0 z + \omega_0 + \cfrac{1}{-l_{1}z + \cfrac{1}{\dip_1 z+\omega_1 + \cfrac{1}{ \;\ddots\;   + \cfrac{1}{-l_{N} z + \cfrac{1}{\tilde{m}_i(z)}}}}}\,, \quad z\in\C\backslash\R,
 \end{align}
   for  some $N\in\N$, non-negative constants $\dip_0,\ldots,\dip_{N-1}$, real constants $\omega_0,\ldots,\omega_{N-1}$ and positive constants $l_1,\ldots,l_N$ with $\dip_n+|\omega_n|>0$ for all $n\in\{1,\ldots,N-1\}$, then 
   \begin{align}\label{eqnmspolyclos}
     m_1(z) - m_2(z) = \oo\biggl(\frac{1}{|z|^{2K+1}}\biggr)
   \end{align}
   as $|z|\rightarrow\infty$ along the imaginary axis, where the integer $K\in\N_0$ is given by 
      \begin{align}
      K = N -1 + \#\{n\in\{1,\ldots,N-1\}\,|\,\dip_n\not=0\} + j 
   \end{align} 
   and $j$ is equal to one when  
   \begin{align}
     \lim_{\eta\rightarrow\infty} \frac{\tilde{m}_1(\I\eta)}{\I\eta} = \lim_{\eta\rightarrow\infty} \frac{\tilde{m}_2(\I\eta)}{\I\eta} >0
   \end{align}
   and zero otherwise. 
 \end{lemma}
 
 \begin{proof}
  For $i\in\{1,2\}$, let us define Herglotz-Nevanlinna functions $q_{i,1},\ldots,q_{i,N}$ by 
  \begin{align*}
    \frac{1}{q_{i,N}(z)} = - l_N z + \frac{1}{\tilde{m}_i(z)}, \quad z\in\C\backslash\R, 
   \end{align*}
  and for $n\in\{1,\ldots,N-1\}$ recursively via  
  \begin{align*}
    \frac{1}{q_{i,n}(z)} = - l_{n} z + \frac{1}{\dip_{n} z + \omega_{n} + q_{i,n+1}(z)}, \quad z\in\C\backslash\R,
  \end{align*}
  so that $m_i=\dip_0 z +\omega_0 + q_{i,1}$ by assumption. 
 A computation then shows that 
 \begin{align*}
    \frac{1}{q_{1,n}(z)} - \frac{1}{q_{2,n}(z)} & = \OO\biggl(\frac{|q_{1,n+1}(z) - q_{2,n+1}(z)|}{|z|^{2j_n}}\biggr), & j_n & = \begin{cases} 0, & \dip_n=0, \\ 1, & \dip_n\not=0, \end{cases}
 \end{align*}
 as $|z|\rightarrow\infty$ along the imaginary axis for all $n\in\{1,\ldots,N-1\}$, which yields 
 \begin{align*}
   q_{1,n}(z) - q_{2,n}(z) &  = \frac{q_{2,n}(z)^{-1}-q_{1,n}(z)^{-1}}{q_{1,n}(z)^{-1}q_{2,n}(z)^{-1}} =  \OO\biggl(\frac{|q_{1,n+1}(z) - q_{2,n+1}(z)|}{|z|^{2+2j_n}}\biggr) 
 \end{align*}
 as $|z|\rightarrow\infty$ along the imaginary axis.  
 Since this entails that 
 \begin{align*}
   m_1(z) - m_2(z) = q_{1,1}(z) - q_{2,1}(z)  = \OO\biggl(\frac{|q_{1,N}(z) - q_{2,N}(z)|}{|z|^{2K -2j}}\biggr)
 \end{align*}
 as $|z|\rightarrow\infty$ along the imaginary axis, it remains to note that 
  \begin{align*}
   q_{1,N}(z) - q_{2,N}(z) = \frac{\tilde{m}_{2}(z)^{-1}-\tilde{m}_{1}(z)^{-1}}{q_{1,N}(z)^{-1}q_{2,N}(z)^{-1}} = \oo\biggl(\frac{1}{|z|^{1+2j}}\biggr)
 \end{align*}
 as $|z|\rightarrow\infty$ along the imaginary axis. 
 \end{proof}

 With the help of these two lemmas, we are now ready to add another item to the equivalence in Theorem~\ref{thmmmomexp}; a continued fraction expansion of the function $m$. 
 
 \begin{theorem}\label{thmmKexp}
   Suppose that $m$ is a non-rational Herglotz--Nevanlinna function and let $K\in\N_0$. 
   Then the moments of the measure $\rho$ exist up to order $2K$ if and only if there is a Herglotz--Nevanlinna function $\tilde{m}$, an integer $N\in\N$, non-negative constants  $\dip_0,\ldots,\dip_N$, real constants $\omega_0,\ldots,\omega_{N-1}$ and positive constants $l_1,\ldots,l_N$ with $\dip_n + |\omega_n|>0$ for all $n\in\{1,\ldots,N-1\}$ and 
   \begin{align}\label{eqnmKcond}
      N - 1 + \#\{n\in\{1,\ldots,N\}\,|\,\dip_n\not=0\} \geq K
   \end{align}
    such that the function $m$ admits the continued fraction expansion 
    \begin{align}\label{eqnmExpK}
   m(z) = \dip_0 z +  \omega_0 + \cfrac{1}{-l_{1}z + \cfrac{1}{\dip_1 z + \omega_1 + \cfrac{1}{ \;\ddots\;   + \cfrac{1}{-l_{N} z + \cfrac{1}{\dip_N z + \tilde{m}(z)}}}}}\,, \quad z\in\C\backslash\R.
 \end{align}
 \end{theorem}

 \begin{proof}
   Assume first that $m$ admits such a continued fraction expansion.
   We define the rational Herglotz--Nevanlinna function $m_1$ by replacing $\tilde{m}$ in the continued fraction on the right-hand side of~\eqref{eqnmExpK} with the function given by 
   \begin{align}\label{eqnReplacemtilde}
       z \lim_{\eta\rightarrow\infty} \frac{\tilde{m}(\I\eta)}{\I\eta} + \biggl(1+z \lim_{\eta\rightarrow\infty} \frac{1}{\I\eta}\frac{1}{\tilde{m}(\I\eta)}\biggr)^{-1}, \quad z\in\C\backslash\R.
   \end{align}
   It then follows from Lemma~\ref{lemHNclose} that 
   \begin{align}\label{eqnmm1asym}
     m(z) =  m_1(z) + \oo\biggl(\frac{1}{|z|^{2K+1}}\biggr)
   \end{align}
   as $|z|\rightarrow\infty$ along the imaginary axis. 
   In view of Theorem~\ref{thmmmomexp}, this already guarantees that the moments of the measure $\rho$ exist up to order $2K$. 

  For the converse direction, we will use induction. 
 To this end, let us first consider the case when $K$ is zero, that is, we assume that the measure $\rho$ is finite, so that   
   \begin{align*}
    -\biggl(\int_\R \frac{1}{\lambda-z}d\rho(\lambda)\biggr)^{-1} = \frac{z}{s_0} + q(z), \quad z\in\C\backslash\R, 
   \end{align*}
   for some Herglotz--Nevanlinna function $q$ with $q(z)/z\rightarrow0$ as $|z|\rightarrow\infty$ along the imaginary axis.
   This allows us to write  
   \begin{align}\label{eqnmexps0}
     m(z) = -s_{-2}z -s_{-1} + \cfrac{1}{-\cfrac{z}{s_0} + \cfrac{1}{\dip_1 z + \tilde{m}(z)}}\,, \quad z\in\C\backslash\R, 
   \end{align}
   for some non-negative constant $\dip_1$ and a Herglotz--Nevanlinna function $\tilde{m}$, which is the claimed expansion. 
  Now let $K\in\N$ and suppose that the assertion holds for all lesser integers.  
  As before, we may write $m$ like in~\eqref{eqnmexps0}  for some non-negative constant $\dip_1$ and a Herglotz--Nevanlinna function $\tilde{m}$ with $\tilde{m}(z)/z\rightarrow0$ as $|z|\rightarrow\infty$ along the imaginary axis. 
  In the case when $K=1$ and $\dip_1\not=0$, this is the required expansion. 
  Otherwise, we conclude from Lemma~\ref{lemMomRel} that the moments of the measure $\tilde{\rho}$ corresponding to the function $\tilde{m}$ exist up to order 
  \begin{align*}
     2(K - 1 - j) & \geq 0, & j& =\begin{cases} 0, & \dip_1=0, \\ 1, & \dip_1\not=0. \end{cases}
  \end{align*}
  Since the function $\tilde{m}$ then admits a continued fraction expansion of the claimed form, we readily see that so does $m$ (it only remains to note that $\dip_1+|\omega_1|>0$ in this expansion since $q(z)/z\rightarrow0$ as $|z|\rightarrow\infty$ along the imaginary axis). 
 \end{proof}
 
  As in the rational case, we are able to provide explicit formulas for the constants in the continued fraction expansion in terms of the Hankel determinants. 
   To this end, let us suppose that $m$ is a non-rational Herglotz--Nevanlinna function such that the moments of the measure $\rho$ exist up to order $2K$ for some $K\in\N_0$. 
  Under these assumptions, the determinants $\Delta_{0,k}$ are well-defined for all $k\in\{0,\ldots,K+1\}$ and positive (as $\rho$ must not be supported on a finite set). 
  The determinants $\Delta_{1,k}$ on the other hand exist at least for $k\in\{0,\ldots,K\}$. 
  We define $N\in\N$ as the number of non-zero elements of the sequence $\Delta_{1,0},\ldots,\Delta_{1,K}$ and introduce the function 
  \begin{align}
    \kappa: \{1,\ldots,N\}\rightarrow\{0,\ldots,K\}
  \end{align}
  such that $\kappa(n)$ is the smallest integer $k\in\{0,\ldots,K\}$ for which the sequence $\Delta_{1,0},\ldots,\Delta_{1,k}$ has precisely $n$ non-zero elements. 
  One observes that the increasing function $\kappa$ is defined in such a way that $\Delta_{1,\kappa(1)},\ldots,\Delta_{1,\kappa(N)}$ enumerates all non-zero members of the sequence $\Delta_{1,0},\ldots,\Delta_{1,K}$. 
  As it follows from relation~\eqref{eqnDeltaRel}  that there are no consecutive zeros in the sequence $\Delta_{1,0},\ldots,\Delta_{1,K}$, we may conclude that 
  \begin{align}
    \kappa(n+1) = \begin{cases} \kappa(n)+1, & \Delta_{1,\kappa(n)+1}\not=0, \\ \kappa(n)+2, & \Delta_{1,\kappa(n)+1}=0, \end{cases}
  \end{align} 
  for all $n\in\{1,\ldots,N-1\}$. 
  Since we have $\kappa(1)=0$, this determines $\kappa$ recursively. 
  
 \begin{corollary}\label{cormKexpForm}
   If $m$ is a non-rational Herglotz--Nevanlinna function such that the moments of the measure $\rho$ exist up to order $2K$ for some $K\in\N_0$, then there is a Herglotz--Nevanlinna function $\tilde{m}$ with 
   \begin{align}
     \lim_{\eta\rightarrow\infty} \frac{1}{\I\eta} \frac{1}{\tilde{m}(\I\eta)} = 0
   \end{align}
   such that the function $m$ admits the continued fraction expansion   
       \begin{align}\label{eqnmExpKCor}
   m(z) = \dip_0 z +  \omega_0 + \cfrac{1}{-l_{1}z + \cfrac{1}{\dip_1 z + \omega_1 + \cfrac{1}{ \;\ddots\;   + \cfrac{1}{-l_{N} z + \cfrac{1}{\tilde{m}(z)}}}}}\,, \quad z\in\C\backslash\R,
 \end{align}
 where the non-negative constants $\dip_0,\ldots,\dip_{N-1}$ and the real constants $\omega_0,\ldots,\omega_{N-1}$ are given by~\eqref{eqnomegadipnDel} 
  and the positive constants $l_1,\ldots,l_{N}$ are given by~\eqref{eqnlnDel}. 
  Furthermore, if the determinant $\Delta_{1,K}$ is zero, then $\kappa(N)=K-1$ and the function $\tilde{m}$ satisfies 
  \begin{align}\label{eqnmtillin}
    \lim_{\eta\rightarrow\infty} \frac{\tilde{m}(\I\eta)}{\I\eta} = \frac{\Delta_{-1,\kappa(N)+2}^2}{\Delta_{0,\kappa(N)+1}\Delta_{0,\kappa(N)+2}}>0.
  \end{align}
 \end{corollary}
 
 \begin{proof}
  Suppose that $m$ is a non-rational Herglotz--Nevanlinna function such that the moments of the measure $\rho$ exist up to order $2K$ for some $K\in\N_0$. 
  According to Theorem~\ref{thmmKexp}, the function $m$ admits a continued fraction expansion of the form~\eqref{eqnmExpK}. 
  With this notation, we define the rational Herglotz--Nevanlinna function $m_1$  by replacing $\tilde{m}$ in the continued fraction on the right-hand side of~\eqref{eqnmExpK} with the function given by~\eqref{eqnReplacemtilde} so that~\eqref{eqnmm1asym} as $|z|\rightarrow\infty$ along the imaginary axis.
  By virtue of the expansion~\eqref{eqnmexpinf} in Theorem~\ref{thmmmomexp}, this shows that the numbers $s_{-2},s_{-1},s_0,\ldots,s_{2K}$ corresponding to the functions $m$ and $m_1$ coincide and hence so do the Hankel determinants $\Delta_{0,0},\ldots,\Delta_{0,K+1}$, $\Delta_{1,0},\ldots,\Delta_{1,K}$, $\Delta_{2,0},\ldots,\Delta_{2,K}$, $\Delta_{-1,0},\ldots,\Delta_{-1,K+1}$ and $\Delta_{-2,0},\ldots,\Delta_{-2,K+2}$.
  In particular, we find that our current function $\kappa$ is a restriction of (or coincides with) the corresponding function~\eqref{eqnkapparat} for $m_1$. 
  Since the continued fraction expansion of the rational function $m_1$ is unique (due to the preconditions that the constants $l_1,\ldots,l_N$ are positive and that $\dip_n+|\omega_n|>0$ for all $n\in\{1,\ldots,N-1\}$), we may obtain expressions for the constants in this expansion in terms of the Hankel determinants by comparison with Proposition~\ref{propmrat}.  
  After possibly redefining the integer $N$ (because the $N$ one gets from Theorem~\ref{thmmKexp} is potentially greater than the number of non-zero elements of the sequence $\Delta_{1,0},\ldots,\Delta_{1,K}$, whereas~\eqref{eqnmKcond} guarantees that it is not less), the constant $l_N$ and the function $\tilde{m}$ in an appropriate way, we see that the expansion~\eqref{eqnmExpK} can be brought into the claimed form.
 \end{proof}

 We now turn to the situation when $m$ is a non-rational Herglotz--Nevanlinna function such that all moments of the measure $\rho$ exist. 
  In this case, all the Hankel determinants are well-defined and the function $\kappa$ extends to an increasing function 
  \begin{align}
    \kappa: \N\rightarrow\N_0
  \end{align}
  such that $\Delta_{1,\kappa(1)},\Delta_{1,\kappa(2)},\ldots$ enumerates all non-zero elements of the sequence $\Delta_{1,0},\Delta_{1,1},\ldots$ (of which there are infinitely many as the sequence does not contain any consecutive zeros).
   
 \begin{corollary}\label{cormexp}
   Suppose that $m$ is a non-rational Herglotz--Nevanlinna function. 
   Then all moments of the measure $\rho$ exist if and only if there are Herglotz--Nevanlinna functions $\tilde{m}_1,\tilde{m}_2,\ldots$, non-negative constants  $\dip_0,\dip_1,\ldots$, real constants $\omega_0,\omega_1,\ldots$ and positive constants $l_1,l_2,\ldots$ with $\dip_n + |\omega_n|>0$ for all $n\in\N$ such that for every $N\in\N$ the function $m$ admits the continued fraction expansion  
    \begin{align}
   m(z) = \dip_0 z +  \omega_0 + \cfrac{1}{-l_{1}z + \cfrac{1}{\dip_1 z + \omega_1 + \cfrac{1}{ \;\ddots\;   + \cfrac{1}{-l_{N} z + \cfrac{1}{\tilde{m}_N(z)}}}}}\,, \quad z\in\C\backslash\R.
 \end{align}
 In this case, the non-negative constants $\dip_0,\dip_1,\ldots$ and the real constants $\omega_0,\omega_1,\ldots$ are given by~\eqref{eqnomegadipnDel} and the positive constants $l_1,l_2,\ldots$ are given by~\eqref{eqnlnDel}. 
 \end{corollary}

 \begin{proof}
 It suffices to refer to Theorem~\ref{thmmKexp} and Corollary~\ref{cormKexpForm}, but let us note that the constants in the continued fraction expansion in~\eqref{eqnmExpKCor} are independent of $N$. 
 \end{proof}
 
 Since the last result in this section will not be needed in the following, we shall state it without a proof. 
 In fact, a convenient way to verify it uses one of the forthcoming results from the next section. 
     
  \begin{corollary}\label{cormexplim}
   Suppose that $m$ is a non-rational Herglotz--Nevanlinna function. 
   Then all moments of the measure $\rho$ exist and
   \begin{align}\label{eqnmexprhocond}
     \rho(\{0\}) = \lim_{k\rightarrow\infty} \frac{\Delta_{0,k+1}}{\Delta_{2,k}}
   \end{align} 
    if and only if there are non-negative constants $\dip_0,\dip_1,\ldots$, real constants $\omega_0,\omega_1,\ldots$ and  positive constants $l_1,l_2,\ldots$ with $\dip_n + |\omega_n|>0$ for all $n\in\N$ such that  
   \begin{align}\label{eqnmlimcontfrac}
     m(z) = \lim_{N\rightarrow\infty} \dip_0 z +  \omega_0 + \cfrac{1}{-l_{1}z + \cfrac{1}{\dip_1 z + \omega_1 + \cfrac{1}{ \;\ddots\;   + \cfrac{1}{-l_{N} z}}}}\,, \quad z\in\C\backslash\R.
   \end{align}
   In this case, the non-negative constants $\dip_0,\dip_1,\ldots$ and the real constants $\omega_0,\omega_1,\ldots$ are given by~\eqref{eqnomegadipnDel} and the positive constants $l_1,l_2,\ldots$ are given by~\eqref{eqnlnDel}. 
 \end{corollary}
  
  %

  \section{Generalized indefinite strings of Stieltjes type}
      
   Let $(L,\omega,\dip)$ be a generalized indefinite string so that $L\in(0,\infty]$, $\omega$ is a real distribution\footnote{We denote with $H^1_{\loc}[0,L)$ and $H^1_{\cc}[0,L)$ the function spaces given by 
   \begin{align*}
     H^1_{\loc}[0,L) & =  \lbrace f\in AC_{\loc}[0,L) \,|\, f'\in L^2_{\loc}[0,L) \rbrace, \\
     H^1_{\cc}[0,L) & = \lbrace f\in H^1_{\loc}[0,L) \,|\, \supp(f) \text{ compact in } [0,L) \rbrace.
  \end{align*}
  The space of distributions $H^{-1}_{\loc}[0,L)$ is defined as the topological dual of $H^1_{\cc}[0,L)$. 
  A distribution $\chi$ in $H^{-1}_{\loc}[0,L)$ is said to be real if $\chi(h)$ is real for each real-valued test function $h\in H^1_{\cc}[0,L)$.} 
      in $H^{-1}_{\loc}[0,L)$ and $\dip$ is a non-negative Borel measure on $[0,L)$. 
   We consider the corresponding spectral problem of the form  
 \begin{align}\label{eqnDEho}
  -f''  = z\, \omega f + z^2 \dip f, 
 \end{align}
 where $z$ is a complex spectral parameter. 
 Of course, this differential equation has to be understood in a distributional sense (we refer to \cite{IndefiniteString} for more details).   
  
  \begin{definition}\label{defSolution}
  A solution of~\eqref{eqnDEho} is a function $f\in H^1_{\loc}[0,L)$ such that 
 \begin{align}
  d_f h(0) + \int_{0}^L f'(x) h'(x) dx = z\, \omega(fh) + z^2 \int_{[0,L)} fh\, d\dip, \quad h\in H^1_{\cc}[0,L),
 \end{align}
 for some complex constant $d_f$. 
 In this case, the constant $d_f$ is uniquely determined and will henceforth always be denoted with $f'(0-)$ for apparent reasons. 
 \end{definition}

 Associated with the generalized indefinite string $(L,\omega,\dip)$ is the corresponding Weyl--Titchmarsh function $m$ defined on $\C\backslash\R$ by 
 \begin{align}
  m(z) =  \frac{\psi'(z,0-)}{z\psi(z,0)},\quad z\in\C\backslash\R,
 \end{align} 
 where $\psi(z,\redot)$ is a non-trivial solution of the differential equation~\eqref{eqnDEho} satisfying  
 \begin{align}
   \int_0^L |\psi'(z,x)|^2 dx + \int_{[0,L)} |z\psi(z,x)|^2 d\dip(x) < \infty
 \end{align}
 and vanishing at the right endpoint when $L$ is finite. 
 The function $m$ is a Herglotz--Nevanlinna function and contains all the spectral information of the differential equation~\eqref{eqnDEho}.
 In fact, the measure $\rho$ in the integral representation~\eqref{eqnmInt} for this function is a spectral measure for an underlying self-adjoint linear relation. 
  All other quantities derived from the function $m$ will be denoted and used in the same way as they were introduced in the previous section. 
  
   It was shown in \cite{IndefiniteString} that the mapping $(L,\omega,\dip)\mapsto m$ establishes a one-to-one correspondence between generalized indefinite strings and Herglotz--Nevanlinna functions.  
  In the following, we are going to use our findings from the last section to characterize those Herglotz--Nevanlinna functions that correspond in this way to generalized indefinite strings that begin with a discrete part. 
   To be more precise, we consider generalized indefinite strings $(L,\omega,\dip)$ of the form 
   \begin{align}
     \omega|_{[0,x_{N+1})} & = \omega_0 \delta_0 + \sum_{n=1}^{N} \omega_n \delta_{x_n}, & \dip|_{[0,x_{N+1})} & = \dip_0 \delta_0 + \sum_{n=1}^{N} \dip_n \delta_{x_n},
   \end{align}
   for some integer $N\in\N_0$, increasing points $x_1,\ldots,x_{N+1}$ in $(0,L)$, real weights $\omega_{0},\ldots,\omega_{N}$ and non-negative weights $\dip_0,\ldots,\dip_N$, where $\delta_x$ denotes the unit Dirac measure centered at a point $x\in[0,L)$. 
   For such coefficients, the solutions $\psi(z,\redot)$ of the differential equation~\eqref{eqnDEho} are piecewise linear on the interval $[0,x_{N+1}]$ with kinks only at the points $x_1,\ldots,x_N$ such that
   \begin{align}
     \psi'(z,x_n-) = \psi'(z,x_n+) + (z\,\omega_n + z^2\dip_n)\psi(z,x_n), \quad z\in\C\backslash\R, 
   \end{align} 
    for every $n\in\{0,\ldots,N\}$, from which we obtain the crucial relation     
   \begin{align}\label{eqnWTrec}
       \frac{\psi'(z,x_n-)}{z\psi(z,x_n)} = \dip_n z + \omega_n + \cfrac{1}{-(x_{n+1}-x_n)z + \cfrac{1}{\cfrac{\psi'(z,x_{n+1}-)}{z\psi(z,x_{n+1})}}}\,, \quad z\in\C\backslash\R,
   \end{align}
   as long as the fractions on the right-hand side are well-defined. 
   Here we have set $x_0$ equal to zero for simplicity, so that the left-hand side of~\eqref{eqnWTrec} becomes the Weyl--Titchmarsh function $m$ when $n$ is zero. 
   In particular, these considerations show that the function $m$ admits a continued fraction expansion of the form~\eqref{eqnmratCF}, and hence is rational, when the coefficients $\omega$ and $\dip$ are supported on a finite set. 
   The converse of this statement holds true as well. 
   
   \begin{proposition}\label{propStrFin}
      If the Weyl--Titchmarsh function $m$ is rational, then the generalized indefinite string $(L,\omega,\dip)$ has the form 
              \begin{align}\label{eqnGISfin}
         \omega & = \omega_0 \delta_0 + \sum_{n=1}^{N} \omega_n \delta_{x_n}, & \dip & = \dip_0 \delta_0 + \sum_{n=1}^{N} \dip_n \delta_{x_n}, 
      \end{align}
       where the length $L$, the increasing points $x_1,\ldots,x_N$ in $(0,L)$, the real weights $\omega_0,\ldots,\omega_N$ and the non-negative weights $\dip_0,\ldots,\dip_N$ are given by 
      \begin{subequations}\label{eqnGISDel}
      \begin{align}
                \frac{1}{L} & = \frac{\Delta_{0,\kappa(N+1)+1}}{\Delta_{2,\kappa(N+1)}}, & x_n & = \frac{\Delta_{2,\kappa(n)}}{\Delta_{0,\kappa(n)+1}},  \\
        	  \omega_0 & = -\frac{\Delta_{-1,1}}{\Delta_{1,0}}, & \omega_n & = \frac{\Delta_{-1,\kappa(n)+1}}{\Delta_{1,\kappa(n)}} - \frac{\Delta_{-1,\kappa(n+1)+1}}{\Delta_{1,\kappa(n+1)}},\\
	 \dip_0 & =  -\frac{\Delta_{-2,1}}{\Delta_{0,0}}, & \dip_n & = \frac{\Delta_{-2,\kappa(n)+2}}{\Delta_{0,\kappa(n)+1}} -  \frac{\Delta_{-2,\kappa(n+1)+1}}{\Delta_{0,\kappa(n+1)}}.
      \end{align}
           \end{subequations}
   \end{proposition}
   
   \begin{proof}
     Suppose that the Weyl--Titchmarsh function $m$ is rational. 
     According to Proposition~\ref{propmrat}, it admits the continued fraction expansion~\eqref{eqnmratCF}. 
     With the constants from this expansion, we define a generalized indefinite string $(\tilde{L},\tilde{\omega},\tilde{\dip})$ such that  
     \begin{align*}
          \tilde{\omega} & = \omega_0 \delta_0 + \sum_{n=1}^{N} \omega_n \delta_{x_n}, & \tilde{\dip} & = \dip_0 \delta_0 + \sum_{n=1}^{N} \dip_n \delta_{x_n}, & \tilde{L} & = \begin{cases} \infty, & r=0, \\ x_N+\frac{1}{r}, & r>0, \end{cases}
      \end{align*}
     where the increasing points $x_1,\ldots,x_N$ are given by  
      \begin{align}\label{eqnDefxn}
       x_n & = \sum_{j=1}^n l_j. 
     \end{align} 
     Since the Weyl--Titchmarsh function $\tilde{m}$ corresponding to $(\tilde{L},\tilde{\omega},\tilde{\dip})$ admits the same continued fraction expansion as $m$ in view of~\eqref{eqnWTrec}, we conclude that it coincides with $m$ and hence so do the generalized indefinite strings. 
     This shows that $(L,\omega,\dip)$ has the claimed form and the explicit formulas follow readily from the ones in Proposition~\ref{propmrat} and Remark~\ref{remCFcoefRexp}.
   \end{proof}
   
  Let us mention here that the expressions for the weights in Proposition~\ref{propStrFin} can also be brought in alternative forms by employing the relations~\eqref{eqnDeltaRel} and~\eqref{eqnDeltaRelm} as in Remark~\ref{remCFcoefRexp}. 
  
 \begin{remark}
  For generalized indefinite strings as in Proposition~\ref{propStrFin}, one has the identities 
    \begin{align}
      \sum_{j=0}^{n-1} \omega_j & = -\frac{\Delta_{-1,\kappa(n)+1}}{\Delta_{1,\kappa(n)}},  \\ 
     \int_0^{x_n} \Wr(x)^2 dx + \sum_{j=0}^{n-1} \dip_j & = -\frac{\Delta_{-2,\kappa(n)+2}}{\Delta_{0,\kappa(n)+1}}, 
    \end{align} 
    where $\Wr\in L^2_{\loc}[0,L)$ is the normalized anti-derivative of the distribution $\omega$ so that 
     \begin{align}
     \omega(h) = - \int_0^L \Wr(x)h'(x)dx, \quad h\in H^1_{\cc}[0,L).
   \end{align} 
   \end{remark}
      
   We now proceed to add another item to the equivalences in Theorem~\ref{thmmmomexp} and Theorem~\ref{thmmKexp}, this time in terms of the corresponding generalized indefinite string. 
   
   \begin{theorem}\label{thmStrK}
    Suppose that the Weyl--Titchmarsh function $m$ is not rational and let $K\in\N_0$.
    Then the moments of the spectral measure $\rho$ exist up to order $2K$ if and only if there is an integer $N\in\N$, increasing points $x_1,\ldots,x_{N}$ in $(0,L)$, real weights $\omega_0,\ldots,\omega_{N-1}$ and non-negative weights $\dip_0,\ldots,\dip_{N}$ with $\dip_n+|\omega_n|>0$ for all $n\in\{1,\ldots,N-1\}$ and 
    \begin{align}
       N - 1 + \#\{n\in\{1,\ldots,N\}\,|\,\dip_n\not=0\} \geq K
    \end{align} 
    such that the generalized indefinite string $(L,\omega,\dip)$ has the form  
    \begin{align}\label{eqnstrK}
      \omega|_{[0,x_N)} & = \omega_0 \delta_0 + \sum_{n=1}^{N-1} \omega_n \delta_{x_n}, & \dip|_{[0,x_N]} & = \dip_0 \delta_0 + \sum_{n=1}^{N} \dip_n \delta_{x_n}.
    \end{align}
   \end{theorem}
  
   \begin{proof}
     Assume first that the moments of the measure $\rho$ exist up to order $2K$ so that the function $m$ admits the continued fraction expansion~\eqref{eqnmExpK} by Theorem~\ref{thmmKexp}.
      With the constants from this expansion, we consider all generalized indefinite strings $(\tilde{L},\tilde{\omega},\tilde{\dip})$ that satisfy 
     \begin{align*}
      \tilde{\omega}|_{[0,x_N)} & = \omega_0 \delta_0 + \sum_{n=1}^{N-1} \omega_n \delta_{x_n}, & \tilde{\dip}|_{[0,x_N)} & = \dip_0 \delta_0 + \sum_{n=1}^{N-1} \dip_n \delta_{x_n},
    \end{align*}     
    where the points $x_1,\ldots,x_N$ are given by~\eqref{eqnDefxn} and it is supposed that $\tilde{L}>x_N$. 
     According to~\eqref{eqnWTrec}, the Weyl--Titchmarsh functions corresponding to these generalized indefinite strings admit continued fraction expansions of the form~\eqref{eqnmExpK}. 
     It then follows from \cite[Theorem~6.1]{IndefiniteString} that there is one such generalized indefinite string such that the continued fraction expansion coincides precisely with the initial one for $m$.
     We conclude that this generalized indefinite string is necessarily the same as $(L,\omega,\dip)$, which guarantees that $(L,\omega,\dip)$ has the claimed form. 
          
     For the converse direction, it remains to notice that the recursion in~\eqref{eqnWTrec} yields the continued fraction expansion~\eqref{eqnmExpK} with the constants $l_1,\ldots,l_N$ given by 
     \begin{align*}
       l_1 & = x_1, & l_n & = x_n - x_{n-1}, \quad n>1,
     \end{align*}
     and the Herglotz--Nevanlinna function $\tilde{m}$ (see \cite[Lemma~7.1]{IndefiniteString}) given by
     \begin{align*}
       \tilde{m}(z) = \frac{\psi'(z,x_N-)}{z\psi(z,x_N)} - \dip_N z, \quad z\in\C\backslash\R,
     \end{align*}
     so that Theorem~\ref{thmmKexp} guarantees the existence of all moments up to order $2K$. 
   \end{proof}
  
   It is again possible to find explicit expressions for the weights and their positions in Theorem~\ref{thmStrK}. 
   These follow readily from the corresponding formulas in Corollary~\ref{cormKexpForm}.
   
   \begin{corollary}\label{corStrKForm}
     If the Weyl--Titchmarsh function $m$ is not rational and such that the moments of the spectral measure $\rho$ exist up to order $2K$ for some $K\in\N_0$, then the generalized indefinite string $(L,\omega,\dip)$ has the form 
         \begin{align}
      \omega|_{[0,x_N)} & = \omega_0 \delta_0 + \sum_{n=1}^{N-1} \omega_n \delta_{x_n}, & \dip|_{[0,x_N)} & = \dip_0 \delta_0 + \sum_{n=1}^{N-1} \dip_n \delta_{x_n},
    \end{align} 
   where the increasing points $x_1,\ldots,x_{N}$ in $(0,L)$, the real weights $\omega_0,\ldots,\omega_{N-1}$ and the non-negative weights $\dip_0,\ldots,\dip_{N-1}$ are given by~\eqref{eqnGISDel}. 
    Furthermore, if the determinant $\Delta_{1,K}$ is zero, then $\kappa(N)=K-1$ and one has  
    \begin{align}
      \dip(\{x_N\}) = \frac{\Delta_{-1,\kappa(N)+2}^2}{\Delta_{0,\kappa(N)+1}\Delta_{0,\kappa(N)+2}}>0. 
    \end{align}
   \end{corollary}
   
   \begin{proof}
    In view of Corollary~\ref{cormKexpForm} and the proof of Theorem~\ref{thmStrK}, it only remains to justify the last claim, which follows from~\eqref{eqnmtillin} upon taking~\cite[Lemma~7.1]{IndefiniteString} into account. 
   \end{proof}
   
   One can infer from Lemma~\ref{lemHNclose} and Theorem~\ref{thmmmomexp} that the moments of the spectral measure only depend on the discrete part of the generalized indefinite string near the left endpoint. 
   In turn, this discrete part is determined by the moments, as we have seen in Corollary~\ref{corStrKForm}.
   
   \begin{remark}
     Theorem~\ref{thmmmomexp} and Corollary~\ref{corStrKForm} show that the asymptotic behavior  
       \begin{align}
         -m(z) = s_{-2}z + s_{-1} + \sum_{k=0}^{2K} \frac{s_k}{z^{k+1}} + \oo\biggl(\frac{1}{|z|^{2K+1}}\biggr)
       \end{align}
       as $|z|\rightarrow\infty$ along the imaginary axis, of the Weyl--Titchmarsh function $m$ uniquely determines the generalized indefinite string $(L,\omega,\dip)$ near the left endpoint. 
       This can be viewed as a variant of local inverse uniqueness results as in \cite{be01, gekima02, gesi00, lawo11, we04} or an instance of the principle that the asymptotics of the Weyl--Titchmarsh function are related to the behavior of the coefficients near the left endpoint \cite{be89, AsymCS, ka73, ka75, wiwo08}. 
   \end{remark}

  We continue with the characterization of those generalized indefinite strings that give rise to spectral measures with finite moments of arbitrary order. 

  \begin{corollary}\label{corStrinf}
        Suppose that the Weyl--Titchmarsh function $m$ is not rational.
    Then all moments of the spectral measure $\rho$ exist if and only if there are increasing points $x_1,x_2,\ldots$ in $(0,L)$, real weights $\omega_0,\omega_1,\ldots$ and non-negative weights $\dip_0,\dip_1,\ldots$ with $\dip_n+|\omega_n|>0$ for all $n\in\N$ such that the generalized indefinite string $(L,\omega,\dip)$ has the form 
    \begin{align}\label{eqnstrinf}
      \omega|_{[0,L_d)} & = \omega_0 \delta_0 + \sum_{n=1}^{\infty} \omega_n \delta_{x_n}, & \dip|_{[0,L_d)} & = \dip_0 \delta_0 + \sum_{n=1}^{\infty} \dip_n \delta_{x_n},  & L_d & = \sup_{n\in\N} x_n. 
    \end{align}
       In this case, the increasing points $x_1,x_2,\ldots$ in $(0,L)$, the real weights $\omega_0,\omega_1,\ldots$ and the non-negative weights $\dip_0,\dip_1,\ldots$ are given by~\eqref{eqnGISDel}.
  \end{corollary}
  
  \begin{proof}
   The claim follows immediately from Theorem~\ref{thmStrK} and Corollary~\ref{corStrKForm}. 
  \end{proof}

  In order to make sure that the Weyl--Titchmarsh function $m$ corresponds to a generalized indefinite string $(L,\omega,\dip)$ whose coefficients are supported on discrete sets, we just need to introduce one additional condition on the spectral data. 

   \begin{corollary}\label{corStieltjesString}
     Suppose that the Weyl--Titchmarsh function $m$ is not rational.
     Then all moments of the spectral measure $\rho$ exist and  
     \begin{align}\label{eqnCondStrL}
       \rho(\{0\}) = \lim_{k\rightarrow\infty} \frac{\Delta_{0,k+1}}{\Delta_{2,k}}
     \end{align} 
     if and only if there are increasing points $x_1,x_2,\ldots$ in $(0,L)$ with $x_n\rightarrow L$, real weights $\omega_0,\omega_1,\ldots$ and non-negative weights $\dip_0,\dip_1,\ldots$ with $\dip_n+|\omega_n|>0$ for all $n\in\N$ such that  the generalized indefinite string $(L,\omega,\dip)$ has the form 
    \begin{align}\label{eqnstrinfplus}
      \omega & = \omega_0 \delta_0 + \sum_{n=1}^{\infty} \omega_n \delta_{x_n}, & \dip & = \dip_0 \delta_0 + \sum_{n=1}^{\infty} \dip_n \delta_{x_n}. 
    \end{align}
    In this case, the increasing points $x_1,x_2,\ldots$ in $(0,L)$, the real weights $\omega_0,\omega_1,\ldots$ and the non-negative weights $\dip_0,\dip_1,\ldots$ are given by~\eqref{eqnGISDel}. 
   \end{corollary}
   
   \begin{proof}
     This is a consequence of Corollary~\ref{corStrinf} and the relation $\rho(\{0\})=L^{-1}$. 
   \end{proof}
      
    As already indicated before, this last result can be used to prove Corollary~\ref{cormexplim} from the previous section, where one should also recall \cite[Proposition~6.2]{IndefiniteString}. 
      
  \begin{remark}
   In conjunction with the solution of the inverse spectral problem for generalized indefinite strings in~\cite{IndefiniteString}, the characterization in Corollary~\ref{corStieltjesString}  gives rise to a solution of the inverse spectral problem for generalized indefinite strings whose coefficients are supported on discrete sets. 
   Since one has explicit formulas for the solution, this also yields a solution of the inverse spectral problem for the class of generalized indefinite strings for which the distribution $\omega$ is supported on a discrete set and the measure $\dip$ vanishes identically. 
   More precisely, these indefinite strings are determined by the additional conditions that 
   \begin{align}
       \lim_{\eta\rightarrow\infty} \frac{m(\I\eta)}{\I\eta} = 0 
   \end{align}
   and that none of the Hankel determinants $\Delta_{1,1},\Delta_{1,2},\ldots$ is zero. 
   \end{remark}

\end{document}